\documentclass[oneside,reqno]{amsart}

\usepackage{hyperref}
\usepackage[totalheight=20 true cm, totalwidth=15 true cm]{geometry}

\usepackage{xypic}

\usepackage{enumerate}
\usepackage{amsthm}
\usepackage{amsmath}
\usepackage{amssymb}
\usepackage{amsfonts}
\usepackage{dsfont}
\usepackage{array}
\usepackage{tabularx}
\usepackage{graphicx}

\usepackage[shortcuts]{extdash}

\newtheorem{theo}{Theorem}[section]
\newtheorem{lemma}[theo]{Lemma}
\newtheorem{prop}[theo]{Proposition}
\newtheorem{cor}[theo]{Corollary}
\newtheorem{defi}[theo]{Definition}

\newtheorem*{theoremA}{Theorem A}

\theoremstyle{definition}
\newtheorem{ex}[theo]{Example}

\newtheorem{rem}[theo]{Remark}

\newcommand{\f}{\phi}

\newcommand{\Gl}{\operatorname{GL}}

\newcommand{\Aut}{\operatorname{Aut}}
\newcommand{\Hom}{\operatorname{Hom}}

\newcommand{\id}{\operatorname{id}}

\newcommand{\alg}{{\operatorname{alg}}}

\newcommand{\Y}{\mathcal{Y}}

\newcommand{\de}{\delta}

\newcommand{\C}{\mathbb{C}}

\newcommand{\Ga}{\mathbb{G}_a}

\newcommand{\Rep}{\operatorname{Rep}}

\newcommand{\rank}{\operatorname{rank}}

\title{Regular singular differential equations and free proalgebraic groups}

\author{Michael Wibmer}
\address{Michael Wibmer, Institute of Analysis and Number Therory, Graz University of Technology, Kopernikusgasse~24, 8010 Graz, Austria, \url{https://sites.google.com/view/wibmer}}
\email{wibmer@math.tugraz.at}

\thanks{This work was supported by the NSF grants DMS-1760212, DMS-1760413, DMS-1760448 and the Lise Meitner grant M 2582-N32 of the Austrian Science Fund FWF}

\date{\today}

\makeatletter
\@namedef{subjclassname@2020}{\textup{2020} Mathematics Subject Classification}
\makeatother



\begin{document}
\maketitle

\begin{abstract}
	We determine the differential Galois group of the family of all regular singular differential equations on the Riemann sphere. It is the free proalgebraic group on a set of cardinality~$|\C|$. 
\end{abstract}


\section{Introduction}

\enlargethispage{20mm}

The following table depicts the beautiful analogy between classical Galois theory and differential Galois theory over the rational function field $\C(x)$. Excellent introductions to these topics can be found in \cite{Szamuely:GaloisGroupsAndFundamentalGroups} and \cite{SingerPut:differential}.

%
%
%
%
%
%
%
%
%

\vspace{5mm}

	\begin{tabularx}{\textwidth} {  >{\centering\arraybackslash}c |  >{\centering\arraybackslash}X 
			| >{\centering\arraybackslash}X }
		&	Classical Galois theory & Differential Galois theory \\ 
		\hline
	1 &	Univariate polynomials over $\C(x)$ & Linear differential equations over $\C(x)$ \\
		\hline
	2 &		Galois extension of $\C(x)$ and their Galois groups & Picard-Vessiot extensions of $\C(x)$ and their differential Galois groups \\ 
			\hline
	3 &	A finite Galois extensions $L$ of $\C(x)$ corresponds to a ramified cover $p\colon X\to\mathbb{P}^1(\C)$. If $S$ is the finite set of branch points and $x_0\in\C$ is not in $S$, then the Galois group of $L/\C(x)$ can be identified with the image of $\pi_1(\mathbb{P}^1(\C)\smallsetminus S,x_0)$ under its action on $p^{-1}(x_0)$.	& Schlesinger's density theorem: Let $S$ be a finite subset of $\mathbb{P}^1(\C)$ and $x_0\in \C$ not in $S$. The differential Galois group of a regular singular linear differential equation with singularities in $S$, can be identified with the Zariski closure of the image of $\pi_1(\mathbb{P}^1(\C)\smallsetminus S,x_0)$ under its action on the local solution space at $x_0$. \\
			\hline
			4 &	For $x_0\in \C$ and $S$ a finite subset of $\mathbb{P}^1(\C)$ not containing $x_0$, the Galois $G$ group of the maximal algebraic extension of $\C(x)$ with ramification only over $S$, is the profinite completion of $\pi_1(\mathbb{P}^1(\C)\smallsetminus S,x_0)$, i.e., $G$ is the free profinite group on a set of cardinality $|S|-1$.
				 & For $x_0\in \C$ and $S$ a finite subset of $\mathbb{P}^1(\C)$ not containing $x_0$, the differential Galois group $G$ of the family of all regular singular differential equations with singularities inside $S$, is the proalgebraic completion of $\pi_1(\mathbb{P}^1(\C)\smallsetminus S,x_0)$, i.e., $G$ is the free proalgebraic group on a set of cardinality $|S|-1$. \\
		\hline
	5 &	Douady's Theorem: The absolute Galois group of $\C(x)$ is the free profinite group on a set of cardinality $|\C|$. & \vspace{3mm} ?	
		
	\end{tabularx}

\vspace{5mm}

The main goal of this article is to fill in the above question mark. On the face of it, it may seem that the appropriate differential analog of Douady's theorem is ``The absolute differential Galois group of $\C(x)$ is the free proalgebraic group on a set of cardinality $|\C|$''. This is in fact a true statement (\cite{BachmayrHarbaterHartmannWibmer:TheDifferentialGaloisGroupOfRationalFunctionField}). However, Douady arrived at 5 via 3 and 4. In this sense, an appropriate differential analog of Douady's theorem should only be concerned with regular singular differential equations. Our main result is the following differential analog of Douady's theorem.

\begin{theoremA}[{Theorem \ref{theo: main}}] \label{theo: main introduction}
	The differential Galois group of the family of all regular singular differential equations over $\C(x)$ is the free proalgebraic group on a set of cardinality $|\C|$.
\end{theoremA} 

At first glance, it might seem that 5 should follow from 4 rather immediately. However, in general, the projective limit of free profinite groups need not be a free profinite group (\cite[Ex.~9.1.14]{RibesZalesskii:ProfiniteGroups}). The question, when a projective limit of free profinite groups is itself free has attracted some attention but does not seem to be fully understood (\cite[Thm. 3.5.15 and Open Question~9.5.2]{RibesZalesskii:ProfiniteGroups}).

To get from 4 to 5 in the proof of Douady's theorem (see \cite[Sec.~3.4]{Szamuely:GaloisGroupsAndFundamentalGroups} or \cite{Douady:DeterminationDunGroupeDeGalois} for the original reference) one uses a compactness argument and that $\pi_1(\mathbb{P}^1(\C)\smallsetminus S,x_0)$ has more or less canonical generators. 
Another ingredient of the proof is that in the free profinite group on a set with $r$ elements, any subset of $r$ topological generators is a basis. As we will show (Example~\ref{ex: counter example}), the corresponding statement fails for free proalgebraic groups. Therefore, genuinely new ideas are needed in the differential case. In particular, we will use a characterization of free proalgebraic groups in terms of embedding problems.

For a finite subset $X$ of $\C$, it is an immediate consequence of the Riemann-Hilbert correspondence, that the differential Galois group of the family of all regular singular differential equations with singularities in $X\cup \{\infty\}$, is the free proalgebraic group on a set of cardinality $|X|$. We generalize this result from finite subsets of $\C$ to arbitrary subsets of $\C$. Indeed, the case $X=\C$ is exactly Theorem~A.

We note that there is also an analogy between differential Galois theory over $\C(x)$ and classical Galois theory over $k(x)$, with $k$ an algebraically closed field of characteristic $p>0$, such that regular singular points correspond to tamely ramified points, while irregular singular points correspond to wildly ramified points (\cite[Section 11.6]{SingerPut:differential}). Based on this analogy, our result may seem more surprising, because the Galois group of the maximal tamely ramified
extension of $k(x)$ with branch locus in a fixed subset $S$ of $\mathbb{P}^1(k)$ is not a free profinite group.

\medskip

We conclude the introduction with an outline of the article. In Section 2 we recall the definition of free proalgebraic groups and the required results concerning differential Galois theory and the Riemann-Hilbert correspondence.
We then study projetive systems of abstract free groups in Section~3. Finally, in the last section the previous results are applied to prove Theorem A. 

\medskip

The author is grateful to David Harbater and Michael Singer for helpful comments.

\section{Preliminaries and Notation}

In this preparatory section we recall the basic definitions and results concerning regular singular differential equations and differential Galois theory. We also review the definition of free proalgebraic groups.

We use ``$=$'' or ``$\simeq$'' to denote canonical isomorphisms and ``$\cong$'' to denote isomorphisms. (As the implied isomorphism in Theorem A is not canonical, it seems worthwhile to make this distinction.)

\subsection{Free proalgebraic groups}

Throughout this article we work over the field $\C$ of complex numbers. We use the term ``algebraic group'' in lieu of ``affine group scheme of finite type over $\C$''. Similarly, a ``proalgebraic group'' is an ``affine group scheme over $\C$''. By a \emph{closed subgroup} of a proalgebraic group, we mean a closed subgroup scheme. Following \cite[Def. 5.5]{Milne:AlgebraicGroupsTheTheoryOfGroupSchemesOfFiniteTypeOverAField} a morphism $G\to H$ of proalgebraic groups is called a \emph{quotient map} if it is faithfully flat. We use $G\twoheadrightarrow H$ to indicate quotient maps.

We begin by recalling the definition of free proalgebraic groups from \cite{Wibmer:FreeProalgebraicGroups}. Let $\Gamma$ be a proalgebraic group and let $X$ be a set. A map $\varphi\colon X\to \Gamma(\C)$ \emph{converges to $1$} if almost all elements of $X$ map to $1$ in any algebraic quotient of $\Gamma$, i.e., for every algebraic group $G$ and every quotient map $\f\colon \Gamma\twoheadrightarrow G$, all but finitely many elements of $X$ map to $1$ under $X\xrightarrow{\varphi}\Gamma(\C)\xrightarrow{\f_\C}G(\C)$.

The following definition is the special case of \cite[Def. 2.18]{Wibmer:FreeProalgebraicGroups}, where $\mathcal{C}$ is the formation of all algebraic groups and $R=k=\C$.

\begin{defi} \label{defi:free proalgebraic group}
	Let $X$ be a set. A proalgebraic group $\Gamma(X)$ together with a map $\iota\colon X\to \Gamma(X)(\C)$ converging to one is called a \emph{free proalgebraic group} on $X$ if $\iota$ satisfies the following universal property. For every proalgebraic group $G$ and every map $\varphi\colon X\to G$ converging to $1$, there exists a unique morphism $\f\colon \Gamma(X)\to G$ of proalgebraic groups such that
	$$\xymatrix{
	X \ar^-\iota[rr] \ar_-\varphi[rd]& & \Gamma(X)(\C) \ar^-{\f_\C}[ld] \\
	& G(\C) &	
	}$$
commutes.
\end{defi}
As the pair $(\iota,\Gamma(X))$ is unique up to a unique isomorphism, we will usually speak of \emph{the} free proalgebraic group $\Gamma(X)$ on $X$.

The map $\iota\colon X\to \Gamma(X)(\C)$ is injective.  In fact, the induced map $F(X)\to \Gamma(X)(\C)$ from the (abstract) free group $F(X)$ on $X$ to $\Gamma(X)(\C)$ is injective (\cite[Lem. 1.6]{Wibmer:SubgroupsOfFreeProalgebraicGroupsAndMatzatsconjecture}). We will therefore in the sequel identify $X$ with a subset of $\Gamma(X)(\C)$ via $\iota$.

\begin{rem} \label{rem: suffices universal for algebraic}
	To verify that $\Gamma(X)$ is the free proalgebraic group on $X$, it suffices to verify the universal property of Definition \ref{defi:free proalgebraic group} for $G$ algebraic (\cite[Rem. 2.19]{Wibmer:FreeProalgebraicGroups}).
\end{rem}

For a proalgebraic group $G$ and a subset $X$ of $G(\C)$, we denote the smallest closed subgroup $H$ of $G$ such that $X\subseteq H(\C)$ with $\langle X\rangle$. In other words, $\langle X\rangle$ is the closed subgroup of $G$ generated by $X$. By \cite[Thm. 2.17]{Wibmer:FreeProalgebraicGroups} we have $\Gamma(X)=\langle X\rangle$.

\medskip

We next explain why the proof of Douady's theorem (as presented in \cite{Douady:DeterminationDunGroupeDeGalois} or \cite[Sec.~3.4]{Szamuely:GaloisGroupsAndFundamentalGroups}) does not have a direct differential analog. This proof uses a result (\cite[Prop.~1]{Douady:DeterminationDunGroupeDeGalois} or \cite[Lem.~3.4.11]{Szamuely:GaloisGroupsAndFundamentalGroups}) attributed to Serre by Douady, stating that, any set of $n$ elements that topologically generates a free profinite group of rank $n$ is a basis. Equivalently, a surjective endomorphism of a free profinite group of finite rank is an isomorphism. The proof of this result uses a counting argument and therefore does not apply in our context, where finite groups are replaced by algebraic groups.
In fact, as shown in the following example, Serre's result does not hold in our context.
%

\begin{ex} \label{ex: counter example}
	Let $X=\{*\}$ be a set with one element $*$. We will show that not every generator of $\Gamma(X)$ is a basis, i.e., there exists a quotient map $\Gamma(X)\to\Gamma(X)$ that is not an isomorphism.

	The free proalgebraic on one element is of the form $\Gamma(X)=\Ga\times D(\C^\times)$. See \cite[Ex.~2.22]{Wibmer:FreeProalgebraicGroups} or \cite[Cor.~16.26]{Sauloy:DifferentialGaloisTheoryThroughRiemannHilbertCorrespondence}. Here $\Ga$ is the additive group and, as in \cite[Ch. IV, \S1, Sec. 1]{DemazureGabriel:GroupesAlgebriques}, for any abelian group $M$, $D(M)$ denotes the diagonalizable proalgebraic group with character group $M$, i.e., $D(M)(T)=\Hom(M,T^\times)$ for any $\C$-algebra $T$. The map $\iota\colon X\to \Gamma(X)(\C)$ is given by $\iota(*)=(1,\id)\in \C\times \Hom(\C^\times,\C^\times)=\Ga(\C)\times D(\C^\times)(\C)$.
	
	As an abelian group, $\C^\times$ is isomorphic to $(\mathbb{Q}/\mathbb{Z})\oplus V$, where $V$ is a $\mathbb{Q}$-vector space of dimension $|\C|$. In particular, there exists an injective endomorphism $\psi \colon \C^\times\to \C^\times$ that is not an isomorphism. Dualizing $\psi$, we find a quotient map $D(\C^\times)\to D(\C^\times)$ that is not an isomorphism. This trivially extends to a quotient map $\Gamma(X)\to\Gamma(X)$ that is not an isomorphism.
%
%
%
%
%
%
%
\end{ex}

We will need the notion of \emph{proalgebraic completion} of an abstract group. See e.g., \cite{BassLubotzkyMagidMozes:TheProalgebraicCompletionOfRigidGroups}. Note that the proalgebraic completion is sometimes also referred to as the \emph{proalgebraic hull} (e.g., in \cite{Sauloy:DifferentialGaloisTheoryThroughRiemannHilbertCorrespondence}) or as the \emph{Hochschild-Mostow group} (in honor of \cite{HochschildMostow:RepresentationsAndRepresentativeFunctionsOfLieGroups}).

\begin{defi}
	Let $F$ be an (abstract) group. The \emph{proalgebraic completion} $F^\alg$ of $F$ is a proalgebraic group equipped with a morphism $F\to F^\alg(\C)$ of groups satisfying the following universal property: If $G$ is a proalgebraic and $F\to G(\C)$ is a morphism of groups, then there exists a unique morphism $\f\colon F^\alg\to G$ of proalgebraic groups such that 
	$$\xymatrix{
		F \ar[rr] \ar[rd]& & F^\alg(\C) \ar^-{\f_\C}[ld] \\
		& G(\C) &	
	}$$
commutes.
\end{defi}

For $X$ a finite (!) set and $F(X)$ the (abstract) free group on $X$, it follows from the universal properties that $F(X)^\alg\simeq \Gamma(X)$. The proalgebraic completion $F^\alg$ of $F$ can be constructed as the fundamental group of the neutral tannakian category of all finite dimensional $\C$-linear representations of $F$.

\subsection{Differential Galois theory} Introductions to this topic can be found in \cite{Magid:LecturesOnDifferentialGaloisTheory,SingerPut:differential},\cite{CrespoHajto:AlgebraicGroupsAndDifferentialGaloisTheory} and \cite{Sauloy:DifferentialGaloisTheoryThroughRiemannHilbertCorrespondence}. We recall the basic definitions and results, introducing our notation for the subsequent sections along the way.

 We fix a differential field $K$ with derivation $\de\colon K\to K$. We assume that the field of constants $K^\de=\{a\in K|\ \de(a)=a\}$ of $K$ is the field $\C$ of complex numbers. We are mainly interested in the case when $K=\C(x)$ is the rational function field in one variable $x$ and $\de=\frac{d}{dx}$. We consider a family $\mathcal{F}=(\de(y)=A_iy)_{i\in I}$
 of linear differential equations indexed by some set $I$, where $A_i\in K^{n_i\times n_i}$ is a square matrix for every $i\in I$.
 
 \begin{defi}
 	A differential field extension $L/K$ with $L^\de=\C$ is a \emph{Picard-Vessiot extension} for $\mathcal{F}$ if there exist matrices $Y_i\in\Gl_{n_i}(L)$ such that $\de(Y_i)=A_iY_i$ for $i\in I$ and $L$ is generated as a field extension of $K$ be all entries of all $Y_i$'s.
 \end{defi}

The $K$-subalgebra $R$ of $L$ generated by all of entries of all $Y_i$'s and $\frac{1}{\det(Y_i)}$'s is a $K$-$\de$-subalgebra of $L$ and called a \emph{Picard-Vessiot ring} for $\mathcal{F}$. 

For a given family $\mathcal{F}$, a Picard-Vessiot extension exists and is unique up to a $K$-$\de$-isomorphism. The \emph{differential Galois group} $G(L/K)$ of the Picard-Vessiot extension $L/K$, or of the family $\mathcal{F}$, is the functor $T\rightsquigarrow \Aut(R\otimes_\C T/K\otimes_\C T)$, from the category of $\C$-algebras to the category groups, where $T$ is considered as a constant differential ring and the automorphisms are required to commute with the derivation. The functor $G(L/K)$ can be represented by a $\C$-algebra, i.e., $G(L/K)$ is a proalgebraic group. 

A Picard-Vessiot extension $L/K$ is \emph{of finite type} if it is the Picard-Vessiot extension for a single differential equation. This is the case if and only of $G(L/K)$ is algebraic.

Since $L$ is the field of fractions of $R$, any $g\in G(L/K)(\C)$ extends uniquely to a $K$-$\de$-automorphism of $L$. For a closed subgroup $H$ of $G=G(L/K)$ we set $$L^H=\{a\in L|\ h(a)=a \ \forall \ h\in H(\C)\}.$$

\begin{theo}[The differential Galois correspondence] \label{theo: differential Galois correspondencee}
Let $L/K$ be a Picard-Vessiot extension. The assignment $M\mapsto G(L/M)$ defines an inclusion reversing bijection between the set of intermediate differential fields of $L/K$ and the set of closed subgroups of $G(L/K)$. The inverse is given by $H\mapsto L^H$. 

If $M$ corresponds to $H$ under this bijection, then $M/K$ is Picard-Vessiot if and only if $H$ is normal in $G(L/K)$. Moreover, if this is the case, the restriction morphism $G(L/K)\to G(M/K)$ is a quotient map with kernel $G(L/M)$. In particular, $G(M/K)\simeq G(L/K)/G(L/M)$.
\end{theo}

An alternative definition of the differential Galois group of a family of linear differential equations can be given via the tannakian formalism (\cite{DeligneMilne:TannakianCategories, Deligne:categoriestannakien}). As a first step, one has to define the ``category of differential equations''. This is formalized through the notion of a differential module. A \emph{differential module} over $K$ is a finite dimensional $K$-vector space $M$ equipped with an additive map $\partial\colon M\to M$ such that $\partial(am)=\de(a)m+a\partial(m)$ for all $a\in K$ and $m\in M$. 
A morphism of differential modules over $K$ is a $K$-linear map that commutes with $\partial$.

To a linear differential equation $\de(y)=Ay$ with $A\in K^{n\times n}$, one associates the differential module $M_A$ by setting $M_A=K^n$ and $\partial(\xi)=\de(\xi)-A\xi$ for all $\xi\in K^n$. Conversely, if $(M,\partial)$ is a differential module with basis $\underline{e}=(e_1,\ldots,e_n)$, we can write $\partial(\underline{e})=\underline{e}(-A)$ for a unique matrix $A=A_{M,\underline{e}}\in K^{n\times n}$.
Then $\partial(\underline{e}\xi)=\underline{e}\de(\xi)+\partial(\underline{e})\xi=\underline{e}(\de(\xi)-A\xi)$ so that $M$ is isomorphic to $M_A$. Via the choice of the basis $\underline{e}$, we can thus associate to $M$ the differential equation $\de(y)=Ay$. A different choice of basis leads to a Gauge equivalent differential equation. 

For a differential module $(M,\partial)$ over $K$ one sets $M^\partial=\{m\in M|\ \partial(m)=0\}$. This is a $\C$\=/subspace of $M$. If $\underline{e}$ is a $K$-basis of $M$ and $A\in K^{n\times n}$ is such that $\partial(\underline{e}\xi)=\underline{e}(\de(\xi)-A\xi)$ for all $\xi\in K^n$, then $M^\partial$ can be identified with the $\C$-space of all solutions of $\de(y)=Ay$ in $K^n$. More generally, if $L$ is a differential field extension of $K$, then $M\otimes_K L$ is a differential module over $L$ and $M\otimes_KL$ identifies with the space of all solutions of $\de(y)=Ay$ in $L^n$.



 The category of all differential modules over $K$ is a neutral Tannakian category over $\C$. 
If $\mathcal{F}=(\de(y)=A_iy)_{i\in I}$ is a family of differential equations over $K$ and $\mathcal{M}=(M_{A_i})_{i\in I}$ is the corresponding family of differential modules, then the differential Galois group of $\mathcal{F}$ is isomorphic to the fundamental group of the neutral Tannakian category $\langle\langle \mathcal{M}\rangle\rangle$ generated by $\mathcal{M}$.

\subsection{Regular singular differential equations}
For background on regular singular differential equations and the Riemann-Hilbert correspondence see \cite[Part 3]{Sauloy:DifferentialGaloisTheoryThroughRiemannHilbertCorrespondence}, \cite[Chapters~5 and~~6]{SingerPut:differential}, \cite[Part I]{MitschiSauzin:DivergentSeriesSummabilityAndResurgenceI} and \cite[Part III]{AndreBaldassarriCailotte:DeRhamcohomologyOfDifferentialModules}.

\medskip

We first treat the local definitions. Consider a differential module $(M,\partial)$ over the field $\C((t))$ of formal Laurent series in $t$  equipped with the usual derivation $\de=\frac{d}{dt}$. Let $\C[[t]]\subseteq \C((t))$ be the differential subring of formal power series. A $\C[[t]]$\=/lattice in $M$ is a $\C[[t]]$\=/submodule $N$ of $M$ such that there exists a $\C[[t]]$-basis of $N$ that is also a $\C((t))$-basis of~$M$. 

The differential module $M$ is called \emph{regular} if there exists a $\C[[t]]$\=/lattice $N$ in $M$ such that $\partial(N)\subseteq N$. 
The differential module $M$ is called \emph{regular singular} if there exists a $\C[[t]]$\=/lattice $N$ in $M$ such that $t\partial(N)\subseteq N$. So a regular differential module is regular singular.

 
\medskip 
 
We now consider the global picture. Let $(M,\partial)$ be differential module over $\C(x)$. Here, as throughout the paper, the rational function field $\C(x)$ is considered as a differential field via the derivation $\de=\frac{d}{dx}$.

Let $\mathbb{P}^1(\C)=\C\cup\{\infty\}$ denote the Riemann sphere. For every point $p\in\mathbb{P}^1(\C)$ we have a ``local'' differential field 
$\C(x)_p=\C((t))$. For $p\in \C$ this is $\C(x)_p=\C((x-p))=\C((t))$, the field of formal Laurent series in $t=x-p$ with derivation $\frac{d}{dt}=\frac{d}{d(x-p)}$. For $p=\infty$ this is $\C(x)_p=\C((x^{-1}))=\C((t))$, the field of formal Laurent series in $t=x^{-1}$ with derivation $\frac{d}{dt}=\frac{d}{dx^{-1}}$. 

Note that $(\C(x),\frac{d}{dx})$ is a differential subfield of $(\C((x-p)),\frac{d}{d(x-p)})$ for $p\in\C$. Thus, we obtain a differential module $M_p=M\otimes_{\C(x)}\C(x)_p$ over $\C(x)_p$ for every $p\in \C$. 
For $p=\infty$ it is not true that $(\C(x),\frac{d}{dx})$ is a differential subfield of $(\C((x^{-1})),\frac{d}{dx^{-1}})$. However, $(\C(x),-x^{2}\frac{d}{dx})$ is a differential subfield of $(\C((x^{-1})),\frac{d}{dx^{-1}})$. So we can base change the differential module $(M,{-x^2}\partial)$ over $(\C(x),-x^{2}\frac{d}{dx})$ to a differential module $M_{\infty}$ over $\C(x)_\infty$. 
This awkwardness at infinity is one of the reasons why some authors prefer to work with connections, rather than differential modules. This way one can avoid the a priori choice of a derivation on $\C(x)$. However, since the module $\Omega_{\C(x)/\C}$ of differentials of $\C(x)$ over $\C$ is one dimensional, these two approaches are equivalent.

 A point $p\in \mathbb{P}^1(\C)$ is a \emph{singularity} of $M$ if the differential module $M_p$ over $\C((t))$ is not regular. 
  A point $p\in \mathbb{P}^1(\C)$ is \emph{regular singular} for $M$ if the differential module $M_p$ over $\C((t))$ is regular singular. Finally, $M$ is called \emph{regular singular} if every point $p\in\mathbb{P}^1(\C)$ is regular singular for $M$.

  Consider a differential equation $\de(y)=Ay$ with $A\in \C(x)^{n\times n}$. A point $p\in \mathbb{P}^1(\C)$ is a \emph{singularity} of $\de(y)=Ay$ if it is a singularity of the associated differential module $M_A$. Note that this definition is at odds with the common terminology, referring to the poles of $A$ as the singularities of $\de(y)=Ay$. If $p\in \C$ is a singularity of $\de(y)=Ay$, then $p$ must be a pole of $A$. However, the converse is not true. A pole of $A$ that is not a singularity of $\de(y)=Ay$ is sometimes called an apparent singularity. The differential equation $\de(y)=Ay$ is called \emph{regular singular} if the associated differential module $M_A$ is regular singular.

\medskip

  Fix a proper subset $X$ of $\C\subseteq \mathbb{P}^1(\C)$. As $X$ is assumed to be a proper subset of $\C$, we can choose a ``base point'' $x_0\in \C$ with $x_0\notin X$.  
     Let $\operatorname{RegSing}(\C(x),X)$ denote the category of all regular singular differential modules over $\C(x)$ with singularities contained in $X\cup\{\infty\}$. This is a Tannakian category over $\C$. A fibre functor $\omega_{X,x_0}\colon  \operatorname{RegSing}(\C(x),X)\to \operatorname{Vec}_\C$, with values in the category $\operatorname{Vec}_\C$ of finite dimensional $\C$-vector spaces, is given by $$\omega_{X,x_0}((M,\partial))=(M\otimes_{\C(x)}\mathcal{M}_{x_0})^\partial,$$ where $\mathcal{M}_{x_0}$ is the differential field of germs of meromorphic functions at $x_0$. We denote with $\underline{\operatorname{Aut}}^\otimes(\omega_{X,x_0})$ the proalgebraic group of tensor automorphisms of $\omega_{X,x_0}$.

Now assume that $X$ is finite. Let $\pi_1(\C\smallsetminus X, x_0)$ be the topological fundamental group of the Riemann sphere with the points $X\cup\{\infty\}$ removed, with base point $x_0$. Then the local solution space $(M\otimes_{\C(x)}\mathcal{M}_{x_0})^\partial$ is naturally equipped with the monodromy action of \mbox{$\pi_1(\C\smallsetminus X, x_0)$}. We denote with \mbox{$\Rep(\pi_1(\C\smallsetminus X, x_0))$} the category of finite dimensional $\C$-linear representations of \mbox{$\pi_1(\C\smallsetminus X,x_0)$}. The following theorem is sometimes referred to as the (global) Riemann-Hilbert correspondence. The essential surjectivity of the functor in the theorem is also known as the solution of the (weak form of the) Riemann-Hilbert problem.

     \begin{theo}[Riemann-Hilbert correspondence] \label{theo: RH}
     	Let $X\subseteq \C$ be finite. Then the functor $$\operatorname{RegSing}(\C(x),X)\to \Rep(\pi_1(\C\smallsetminus X,x_0)),\quad  (M,\partial)\rightsquigarrow  (M\otimes_{\C(x)}\mathcal{M}_{x_0})^\partial$$
     	is an equivalence of Tannakian categories.
     \end{theo}

From Theorem \ref{theo: RH} we immediately obtain:
 
 \begin{cor} \label{cor: RH finite}
 		For $X\subseteq \C$ finite we have
 	$\underline{\operatorname{Aut}}^{\otimes}(\omega_{X,x_0})\simeq \pi_1(\C\smallsetminus X,x_0)^\alg$. In particular, the differential Galois group of the family of all regular singular differential equations over $\C(x)$ with singularities in $X\cup\{\infty\}$ is isomorphic to the free proalgebraic group on $X$.
 \end{cor}
\begin{proof}
	As $\pi_1(\C\smallsetminus X,x_0)^\alg$ is the proalgebraic group of tensor automorphisms of the forgetful functor $\Rep(\pi_1(\C\smallsetminus X,x_0))\to \operatorname{Vec}_\C$, Theorem \ref{theo: RH} yields an isomorphism $\pi_1(\C\smallsetminus X,x_0)^\alg\to \underline{\operatorname{Aut}}^{\otimes}(\omega_{X,x_0})$. 
	The last statement follows because the differential Galois group of the family of all regular singular differential equations with singularities in $X\cup\{\infty\}$ is isomorphic to $\underline{\operatorname{Aut}}^{\otimes}(\omega_{X,x_0})$ and the group $\pi_1(\C\smallsetminus X,x_0)$ is free on $|X|$ generators.
\end{proof}
 
 We stress the fact that the morphism of groups $\pi_1(\C\smallsetminus X,x_0)\to \underline{\operatorname{Aut}}^{\otimes}(\omega_{X,x_0})(\C)$ is canonical. 
 If $X'$ is a subset of the finite set $X\subseteq \C$, then $\operatorname{RegSing}(\C(x),X')$ is a subcategory of $\operatorname{RegSing}(\C(x),X)$ and $\omega_{X',x_0}$ is the restriction of $\omega_{X,x_0}$ to $\operatorname{RegSing}(\C(x),X')$. We thus have a morphism $\underline{\operatorname{Aut}}^{\otimes}(\omega_{X,x_0})\to \underline{\operatorname{Aut}}^{\otimes}(\omega_{X',x_0})$ of proalgebraic groups. As $\C\smallsetminus X\subseteq \C\smallsetminus X'$, we also have a morphism $\pi_1(\C\smallsetminus X, x_0)\to \pi_1(\C\smallsetminus X', x_0)$ of groups. The diagram
 $$
 \xymatrix{
 \pi_1(\C\smallsetminus X,x_0) \ar[d] \ar[r] &	\underline{\operatorname{Aut}}^{\otimes}(\omega_{X,x_0})(\C) \ar[d] \\
 \pi_1(\C\smallsetminus X',x_0) \ar[r] &	\underline{\operatorname{Aut}}^{\otimes}(\omega_{X',x_0})(\C) 
 	 } 
 $$
 commutes and so also the diagram

  \begin{equation} \label{eq: can isom between systems}
 \xymatrix{
 	\pi_1(\C\smallsetminus X,x_0)^\alg \ar[d] \ar^-{\simeq}[r] &	\underline{\operatorname{Aut}}^{\otimes}(\omega_{X,x_0}) \ar[d] \\
 	\pi_1(\C\smallsetminus X',x_0)^\alg \ar^-{\simeq}[r] &	\underline{\operatorname{Aut}}^{\otimes}(\omega_{X',x_0})
 } 
 \end{equation}
 commutes.
 
 The family of isomorphisms $\pi_1(\C\smallsetminus X,x_0)^\alg\simeq 	\underline{\operatorname{Aut}}^{\otimes}(\omega_{X,x_0})$, one for every finite subset $X$ of $\C$ not containing $x_0$, can thus be seen as defining an isomorphism between two projective systems of proalgebraic groups. The projective limit on the right hand side is $\varprojlim_{X}\underline{\operatorname{Aut}}^{\otimes}(\omega_{X,x_0})=\underline{\operatorname{Aut}}^{\otimes}(\omega_{\C\smallsetminus\{x_0\},x_0})$, whereas the projective limit $\varprojlim_{X}\pi_1(\C\smallsetminus X,x_0)^\alg$ on the left hand side is a projective limit of free proalgebraic groups. The following section provides the necessary tools to show that this limit itself is free.

\section{Projective systems of free groups}
\label{sec: Projective systems}

Let $X$ be a set and consider the directed set $\mathcal{Y}$ of all finite subsets of $X$ ordered by inclusion. For $Y\in\mathcal{Y}$ let $F(Y)$ denote the (abstract) free group on $Y$ and for $Y\subseteq Y'$, define a map $\varphi_{Y,Y'}\colon F(Y')\to F(Y)$
by
$$\varphi_{Y,Y'}(y')=
\begin{cases}
y' &  \text{ if } y'\in Y, \\
1 &  \text{ if } y'\notin Y.
\end{cases}
$$
The projective limit $\varprojlim_{Y\in \mathcal{Y}}F(Y)$ (in the category of groups) of the projective system 
$(F(Y)_{Y\in \Y}, (\varphi_{Y,Y'})_{Y'\supseteq Y})$ is in general not a free group: We have a map $\varphi\colon X\to \varprojlim_{Y\in \mathcal{Y}}F(Y),\ x\mapsto (\varphi(x)_Y)_{Y\in\mathcal{Y}}$ given by $$\varphi(x)_Y=\begin{cases}
x &  \text{ if } x\in Y, \\
1 &  \text{ if } x\notin Y
\end{cases}$$
and the induced map $F(X)\to  \varprojlim_{Y\in \mathcal{Y}}F(Y)$ is injective. However, this map need not be surjective. Intuitively, surjectivity fails because $F(X)$ only contains words of finite length, while  $\varprojlim_{Y\in \mathcal{Y}}F(Y)$ may contain words of infinite length. For example, if $X=\{x_1,x_2,\ldots\}$ is countably infinite, write $Y=\{x_{i_{Y,1}},x_{i_{Y,2}},\ldots,x_{i_{Y,|Y|}}\}$ with $i_{Y,1}<i_{Y,2},\ldots<i_{Y,|Y|}$. Then $(x_{i_{Y,1}}x_{i_{Y,2}}\ldots x_{i_{Y,|Y|}})_{Y\in\mathcal{Y}}$ lies in $\varprojlim_{Y\in \mathcal{Y}}F(Y)$ but not in the image of $F(X)$. 
In the nomenclature of \cite{Higman:UnrestrictedFreeProducts} this limit is an \emph{unrestricted free product}. For more on projective limits of abstracts free groups see \cite{ConnerKent:InverseLimitsOfFiniteRankFreeGroups} and \cite{EdaNakamura:TheClassificationOfTheInverseLimitsOfSequencesOfFreeGroupsOfFiniteRank}.

However, when working with free profinite groups instead of abstract free groups, the above construction leads to a free profinite group. See \cite[Cor. 3.3.10 b)]{RibesZalesskii:ProfiniteGroups} or \cite[Lem. 3.4.10]{Szamuely:GaloisGroupsAndFundamentalGroups}. As we shall now explain, also in the case of free proalgebraic groups, the above construction leads to a free proalgebraic group.

As above, let $X$ be a set and let $\Y$ be the directed set of all finite subsets of $X$. For $Y\subseteq Y'$, by the universal property of $\Gamma(Y')$ (Definition \ref{defi:free proalgebraic group}), the map $\varphi_{Y,Y'}\colon Y'\to \Gamma(Y)(\C)$
defined by
$$\varphi_{Y,Y'}(y')=
\begin{cases}
y' &  \text{ if } y'\in Y, \\
1 &  \text{ if } y'\notin Y,
\end{cases}
$$
extends to a morphism $\varphi_{Y,Y'}\colon \Gamma(Y')\to\Gamma(Y)$. Then $((\Gamma(Y))_{Y\in \Y}, (\varphi_{Y,Y'})_{Y'\supseteq Y})$ is a projective system of free proalgebraic groups on finite sets. The projective limit is $\Gamma(X)$.

\begin{lemma} \label{lemma: canonical limit}
	$\varprojlim_{Y\in \mathcal{Y}}\Gamma(Y)=\Gamma(X)$.
\end{lemma}
\begin{proof}
	For $Y\in\Y$ we define $\iota_Y\colon X\to \Gamma(Y)(\C)$ by $\iota_Y(x)=\begin{cases} x \text{ if } x\in Y, \\
	1 \text{ if } x\notin Y. \end{cases}$ Then $$\iota\colon X\to \varprojlim_{Y\in \mathcal{Y}}\Gamma(Y)(\C)=\Big(\varprojlim_{Y\in \mathcal{Y}}\Gamma(Y)\Big)(\C), \quad x\mapsto  (\iota_Y(x))_{Y\in\Y}$$
	converges to $1$ because every quotient map $\varprojlim_{Y\in \mathcal{Y}}\Gamma(Y)\twoheadrightarrow G$ to an algebraic group $G$ factors through some $\Gamma(Y)$.
	
	To verify the universal property of $\iota$, according to Remark \ref{rem: suffices universal for algebraic}, it suffices to consider a map $\varphi\colon X\to G(\C)$ converging to $1$, with $G$ an algebraic group. Then $Y_0=\{x\in X|\ \varphi(x)\neq 1\}$ is finite. Define a morphism $\psi\colon \Gamma(Y_0)\to G$ by $\psi(y)=\varphi(y)$ for $y\in Y_0$ and let $\f$ be the composition $\f\colon \varprojlim_{Y\in \mathcal{Y}}\Gamma(Y)\to \Gamma(Y_0)\xrightarrow{\psi} G$. Then $\f(\iota(x))=\varphi(x)$ for every $x\in X$.
	
	If $\f'\colon \varprojlim_{Y\in \mathcal{Y}}\Gamma(Y)\to G$ is another morphism such that $\f'(\iota(x))=\varphi(x)$ for all $x\in X$, then 
	$T\rightsquigarrow\{g\in (\varprojlim_{Y\in \mathcal{Y}}\Gamma(Y))(T)|\ \f(g)=\f'(g)\}$ is a closed subgroup of $\varprojlim_{Y\in \mathcal{Y}}\Gamma(Y)$. It thus suffices to show that $\langle\iota(X)\rangle=\varprojlim_{Y\in \mathcal{Y}}\Gamma(Y)$. The projections $\langle \iota(X)\rangle\to \Gamma(Y)$ are quotient maps because $\Gamma(Y)=\langle Y\rangle$. But then also $\langle \iota(X)\rangle\to \varprojlim_{Y\in \mathcal{Y}}\Gamma(Y)$ is a quotient map. Therefore $\langle\iota(X)\rangle=\varprojlim_{Y\in \mathcal{Y}}\Gamma(Y)$.
\end{proof}


We now specialize to the case that $X$ is a proper subset of $\C$. In this case, besides the projective system $((F(Y))_{Y\in\Y}, (\varphi_{Y',Y})_{Y'\supseteq Y})$ from the beginning of this section, we can associate another projective system of finite rank (abstract) free groups to $X$ as follows. Fix $x_0\in \C$ with $x_0\notin X$. For $Y=\{y_1,\ldots,y_n\}\in \Y$, it is well known (and follows from Van Kampen's theorem) that the fundamental group $\pi_1(\C\smallsetminus Y,x_0)$ of $\C\smallsetminus Y$ with base point $x_0$ is isomorphic to the free group on $n$ generators. A free set of generators is given by choosing, for each $i=1,\ldots,n$, a loop based at $x_0$ that passes once around $y_i$ counterclockwise and does not enclose any other points of $Y$.

For $Y, Y'\in \Y$ with $Y\subseteq Y'$, the inclusion $\C\smallsetminus Y'\subseteq \C\smallsetminus Y$ gives rise to a morphism $$\psi_{Y,Y'}\colon\pi_1(\C\smallsetminus Y',x_0)\to \pi_1(\C\smallsetminus Y,x_0)$$ of groups. In fact, $((\pi_1(\C\smallsetminus Y,x_0))_{Y\in\Y}, (\psi_{Y,Y'})_{Y'\supseteq Y})$ is a projective system of groups.

%
%
%
\begin{lemma} \label{lemma: projective systems isomorphic}
	The projective systems $\big((\pi_1(\C\smallsetminus Y,x_0))_{Y\in\Y}, (\psi_{Y,Y'})_{Y'\supseteq Y}\big)$ and $\big((F(Y))_{Y\in\Y}, (\varphi_{Y,Y'})_{Y'\supseteq Y}\big)$ are isomorphic, i.e., there exists a family $(\alpha_Y)_{Y\in\Y}$ of isomorphisms $\alpha_Y\colon F(Y)\to \pi_1(\C\smallsetminus Y, x_0)$ such that
	\begin{equation}\label{eq: diag for isom}
	\xymatrix{
	F(Y') \ar^-{\alpha_{Y'}}[r] \ar_-{\varphi_{Y',Y}}[d]  & \pi_1(\C\smallsetminus Y', x_0) \ar^-{\psi_{Y,Y'}}[d] \\
	F(Y) \ar^-{\alpha_{Y}}[r] & \pi_1(\C\smallsetminus Y, x_0) 
}
	\end{equation}
	commutes for every $Y'\supseteq Y$.
\end{lemma}
\begin{proof}
	For $Y\in \Y$ and $y\in Y$, we define a \emph{canonical generator at $y$} to be an element of $\pi_1(\C\smallsetminus Y,x_0)$ that is the homotopy class of a loop based at $x_0$ that passes once counterclockwise around $y$, not enclosing any other points of $Y$.
The following graphics depict two canonical generators at $y_1$.

\begin{figure}[h!]
\includegraphics[scale=0.2]{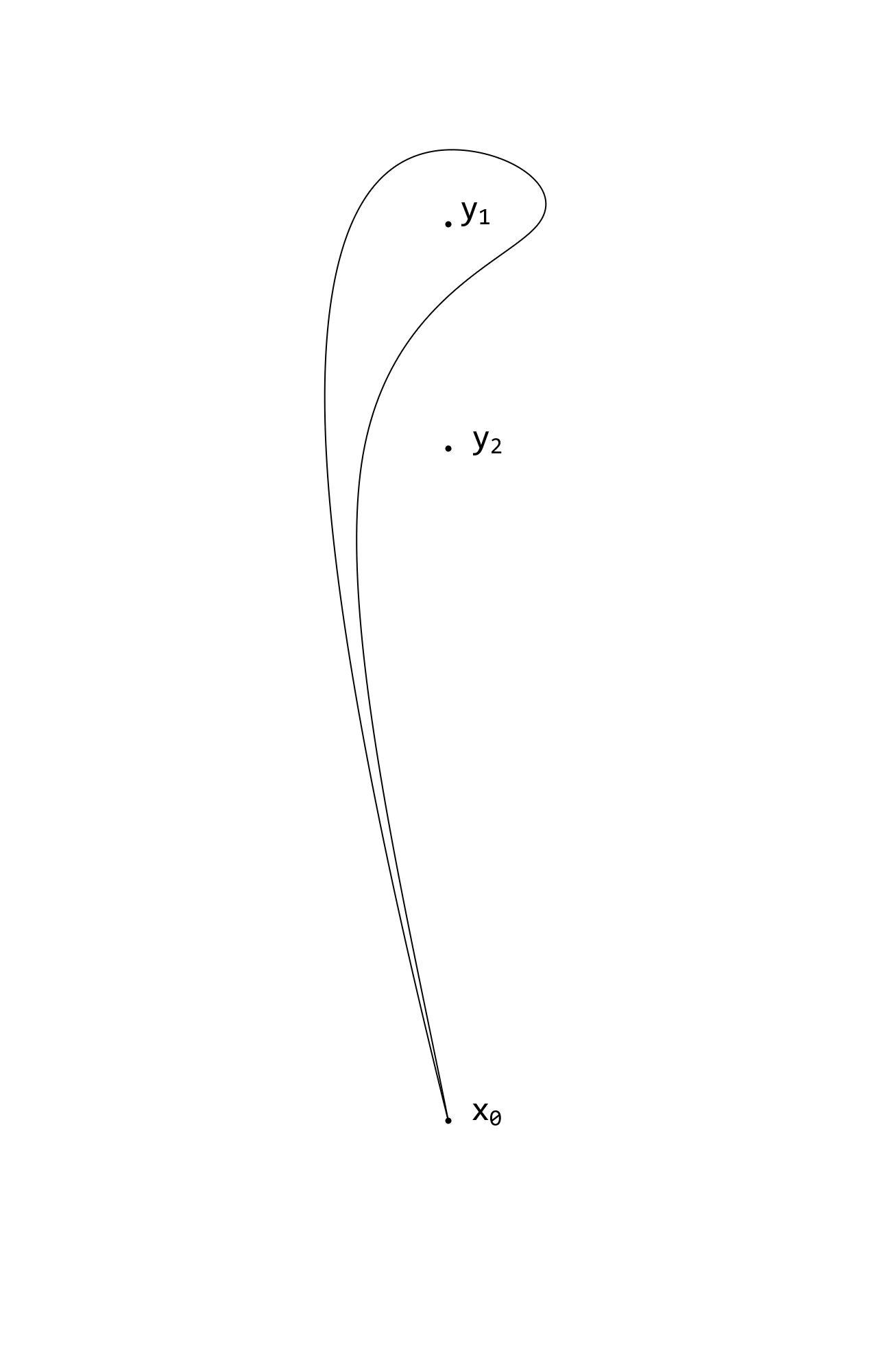}
\includegraphics[scale=0.172]{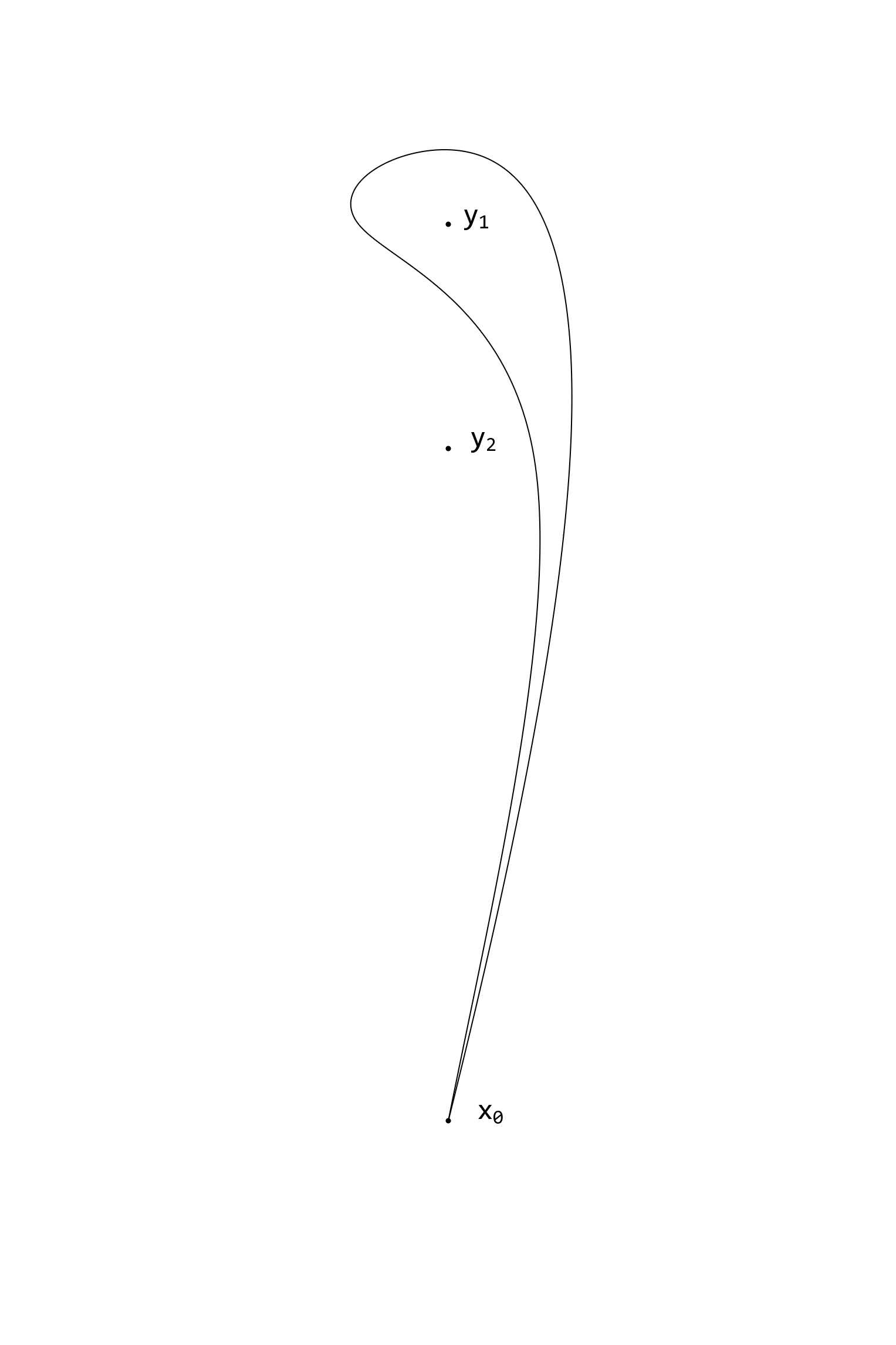}
\end{figure}
	Note that  $\pi_1(\C\smallsetminus Y,x_0)$ contains only finitely many canonical generators at $y$.
	Thus the set	$$B_Y=\{(g_y)_{y\in Y}|\ g_y\in \pi_1(\C\smallsetminus Y, x_0) \text{ is a canonical generator at } y \text{ for every } y\in Y \}$$
	is finite. The map $\psi_{Y,Y'}\colon\pi_1(\C\smallsetminus Y',x_0)\to \pi_1(\C\smallsetminus Y,x_0)$ maps a canonical generator at $y'\in Y'\supseteq Y$ either to a canonical generator at $y$ (if $y'\in Y$) or to $1$ if $y'\notin Y$. 
%
%
%
%

	For $Y, Y'\in \Y$ with $Y\subseteq Y'$, we can define a map $\Psi_{Y,Y'}\colon B_{Y'}\to B_Y$, by $\Psi_{Y,Y'}((g_{y'})_{y'\in Y'})=(\psi_{Y,Y'}(g_y))_{y\in Y}$.
	
	Then $\big((B_Y)_{Y\in\Y}, (\Psi_{Y,Y'})_{Y'\supseteq Y}\big)$ is a projective system of finite sets. Thus the corresponding projective limit is non-empty (\cite[Prop. 1.1.4]{RibesZalesskii:ProfiniteGroups}).
	Let $((g_y)_{y\in Y})_{Y\in\Y}$ be an element of $\varprojlim_{Y\in \mathcal{Y}} B_Y$. For $Y\in\Y$ define $\alpha_Y\colon F(Y)\to \pi_1(\C\smallsetminus Y, x_0)$ by $\alpha_Y(y)=g_y$ for $y\in Y$. Then $\alpha_Y$ is an isomorphism and by construction diagram (\ref{eq: diag for isom}) commutes.
\end{proof}

%
%
%

\section{Main result} \label{sec:main result}

Throughout Section \ref{sec:main result} we will use the following notation.
We fix a Picard-Vessiot extension $L/\C(x)$ for the family all regular singular differential equations over $\C(x)$. For a subset $X$ of $\C$ we denote with $L_X\subseteq L$ the Picard-Vessiot extension of $\C(x)$ for the family of all regular singular differential equations over $\C(x)$ with singularities in $X\cup \{\infty\}$. So $L=L_\C$. We also set $\Gamma_X=G(L_X/\C(x))$.
With this notation, our goal is to show that $\Gamma_X\cong \Gamma(X)$ for every subset $X$ of $\C$.

We first tackle the case of proper subsets of $\mathbb{C}$.
The following proposition generalizes Corollary~\ref{cor: RH finite} from finite subsets of $\C$ to arbitrary proper subsets of $\C$.

\begin{prop} \label{prop: main for X}
	Let $X$ be a proper subset of $\C$. Then the differential Galois group $\Gamma_X$ of the family of all regular singular differential equations over $\C(x)$ with singularities contained in $X\cup\{\infty\}$, is isomorphic to the free proalgebraic group $\Gamma(X)$ on $X$.
\end{prop}
\begin{proof}
	Fix a base point $x_0\in\C$ with $x_0\notin X$. 
	As $\Gamma_X$ is isomorphic to $\underline{\operatorname{Aut}}^\otimes(\omega_{X,x_0})$, it suffices to show that $\underline{\operatorname{Aut}}^\otimes(\omega_{X,x_0})$ is isomorphic to $\Gamma(X)$. (Note that the isomorphism between $\Gamma_X$ and $\underline{\operatorname{Aut}}^\otimes(\omega_{X,x_0})$ is arguably not canonical. It depends on an isomorphism between the fibre functor $\omega_{X,x_0}$ and the fibre functor defined by $L_X$.)
	
	As in Section \ref{sec: Projective systems}, we consider the directed set $\Y$ of all finite subsets of $X$. Because the category $\operatorname{RegSing}(\C(x),X)$ is the union of the subcategories $\operatorname{RegSing}(\C(x),Y)$, where $Y$ runs through all elements of $\Y$, it is clear that $\varprojlim_{Y\in \mathcal{Y}}\underline{\operatorname{Aut}}^\otimes(\omega_{Y,x_0})\simeq \underline{\operatorname{Aut}}^\otimes(\omega_{X,x_0})$.
	The projective systems $(\underline{\operatorname{Aut}}^\otimes(\omega_{Y,x_0}))_{Y\in \Y}$ and 
	$(\pi_1(\C\smallsetminus Y,x_0)^\alg)_{Y\in\Y}$ are canonically isomorphic by (\ref{eq: can isom between systems}). So also $ \varprojlim_{Y\in \mathcal{Y}}\underline{\operatorname{Aut}}^\otimes(\omega_{Y,x_0}) \simeq\varprojlim_{Y\in \mathcal{Y}}\pi_1(\C\smallsetminus Y,x_0)^\alg$.

	By Lemma \ref{lemma: projective systems isomorphic}, there exists an isomorphism between the projective systems (of abstract groups)
	 $\big((\pi_1(\C\smallsetminus Y,x_0))_{Y\in\Y}, (\varphi_{Y',Y})_{Y\subseteq Y'}\big)$ and $\big((F(Y))_{Y\in\Y}, (\psi_{Y,Y'})_{Y\subseteq Y'}\big)$.
	 As the proalgebraic completion defines a functor from the category of groups to the category of proalgebraic groups, we obtain an isomorphism between the projective systems (of proalgebraic groups)
	  $\big((\pi_1(\C\smallsetminus Y,x_0)^\alg)_{Y\in\Y}, (\varphi_{Y',Y}^\alg)_{Y\subseteq Y'}\big)$ and $\big((F(Y)^\alg)_{Y\in\Y}, (\psi^\alg_{Y,Y'})_{Y\subseteq Y'}\big)$. This isomorphism in turn yields an isomorphism $\varprojlim_{Y\in \mathcal{Y}}\pi_1(\C\smallsetminus Y,x_0)^\alg\cong \varprojlim_{Y\in \mathcal{Y}}F(Y)^\alg$ between the corresponding projective limits. The latter limit is, by Lemma \ref{lemma: canonical limit}, isomorphic to $\Gamma(X)$.	 
	 
	 In summary, we have
	 $$\Gamma_X\cong\underline{\operatorname{Aut}}^\otimes(\omega_{X,x_0})\simeq  \varprojlim_{Y\in \mathcal{Y}}\underline{\operatorname{Aut}}^\otimes(\omega_{Y,x_0}) \simeq\varprojlim_{Y\in \mathcal{Y}}\pi_1(\C\smallsetminus Y,x_0)^\alg\cong \varprojlim_{Y\in \mathcal{Y}}F(Y)^\alg\simeq\varprojlim_{Y\in \mathcal{Y}} \Gamma(Y)\simeq \Gamma(X).$$
\end{proof}

Note that Proposition \ref{prop: main for X} and its proof does not apply to the case $X=\C$ of prime interest, because for $X=\C$ we cannot choose a base point $x_0\notin \C\smallsetminus X$. To accomplish the case $X=\C$, we will use a characterization of free proalgebraic groups in terms embedding problems.

To this end, we need to recall some definitions from \cite{Wibmer:FreeProalgebraicGroups}.
Let $\Gamma$ be a proalgebraic group. An \emph{embedding problem} for $\Gamma$ consists of two quotient maps $\alpha\colon G\twoheadrightarrow H$ and $\beta\colon \Gamma\twoheadrightarrow H$ of proalgebraic groups. The embedding problem is \emph{algebraic} if $G$ (and therefore also $H$) is an algebraic group. The embedding problem is trivial if $\alpha$ is an isomorphism. A \emph{solution} of the embedding problem is a quotient map $\f\colon \Gamma \twoheadrightarrow G$ such that
$$
\xymatrix{
\Gamma \ar@{->>}^-\beta[rd] \ar@{..>>}_-\f[d] & \\
G \ar@{->>}^-\alpha[r] & H	
}
$$
commutes. A family $(\f_i)_{i\in I}$ of solutions is \emph{independent} if the induced map $$\prod_{i\in I}\f_i\colon \Gamma\to \prod_{i\in I} (G\twoheadrightarrow H)$$ is a quotient map. Here $\prod_{i\in I} (G\twoheadrightarrow H)$ denotes the fibre product of $|I|$-copies of $G$ over $H$.


The \emph{rank} $\rank(\Gamma)$ of a proalgebraic group $\Gamma$ is defined as the smallest cardinal $\kappa$ such that $\Gamma$ can be written as a projective limit of algebraic groups over an directed set of cardinality $\kappa$. See \cite[Prop. 3.1]{Wibmer:FreeProalgebraicGroups} for other characterizations of the rank.

The following theorem provides a characterization of free proalgebraic groups in terms of algebraic embedding problems.

\begin{theo} \label{theo: characterize freeness through embedding problems}
	Let $\Gamma$ be a proalgebraic group with $\rank(\Gamma)\leq |\C|$.
	 Then $\Gamma$ is isomorphic to the free proalgebraic group on a set of cardinality $|\C|$ if and only if every non-trivial algebraic embedding problem for $\Gamma$ has $|\C|$ independent solutions.
\end{theo}
\begin{proof}
	First assume that $\Gamma$ is isomorphic to $\Gamma(X)$, where $X$ is a set of cardinality $|\C|$. Since $\rank(\Gamma(X))=|X|=|\C|$ by \cite[Cor. 3.12]{Wibmer:FreeProalgebraicGroups}, we find $\rank(\Gamma)=|\C|$. Thus the claim follows from Theorem 3.42 (applied with $\mathcal{C}$ the formation of all algebraic groups) paired with Definition~3.25 of \cite{Wibmer:FreeProalgebraicGroups}.

	Conversely, assume that every non-trivial algebraic embedding problem for $\Gamma$ has $|\C|$ independent solutions. We claim that $\rank(\Gamma)=|\C|$. Consider the embedding problem $\alpha\colon \Ga\to 1$, $\beta\colon \Gamma\to 1$ for $\Gamma$, where $\Ga$ is the additive group. By assumption, there exist solutions $(\f_x)_{x\in\C}$ such that the induced morphism $\Gamma\to \Ga^{|\C|}$ is a quotient map. The rank of $\Ga^{|\C|}$ is $|\C|$ (\cite[Ex. 3.3]{Wibmer:FreeProalgebraicGroups}) and the rank can only decrease when passing to a quotient (\cite[Lem. 3.5]{Wibmer:FreeProalgebraicGroups}). So $\rank(\Gamma)\geq |\C|$ and consequently $\rank(\Gamma)=|\C|$. Our assumption on $\Gamma$ therefore implies that $\Gamma$ satisfies condition (vii) of \cite[Thm. 3.24]{Wibmer:FreeProalgebraicGroups}. Thus the claim follows again from \cite[Thm. 3.42]{Wibmer:FreeProalgebraicGroups}.
\end{proof}

The following lemma is our crutch to go from proper subsets of $\C$ to all of $\C$.

\begin{lemma} \label{lemma: step for main theorem}
	Let $\Gamma$ be a proalgebraic group with $\rank(\Gamma)\leq |\C|$ and let $X$ be a set of cardinality $|\C|$. Assume that every quotient map $\Gamma\twoheadrightarrow H$ to an algebraic group $H$ can be factored as $\Gamma\twoheadrightarrow \Gamma(X)\twoheadrightarrow H$. Then $\Gamma$ is isomorphic to $\Gamma(X)$.
\end{lemma}
\begin{proof}
	 Let
	\begin{equation} \label{diagram: embedding problem}
	\xymatrix{
		\Gamma \ar@{->>}^-\beta[rd] \ar@{..>>}[d] & \\
		G \ar@{->>}^-\alpha[r] & H	
	}
	\end{equation}
	be an non-trivial algebraic embedding problem for $\Gamma$. By assumption, $\beta\colon \Gamma\twoheadrightarrow H$ factors as $\beta\colon\xymatrix{\Gamma \ar@{->>}^-{\beta_0}[r] & \Gamma(X) \ar@{->>}^{\beta'}[r] & H}$.	
	The embedding problem
		\begin{equation*} 
	\xymatrix{
		\Gamma(X) \ar@{->>}^-{\beta'}[rd] \ar@{..>>}[d] & \\
		G \ar@{->>}^-\alpha[r] & H	
	}
	\end{equation*} has $|\C|$ independent solutions $(\f'_x)_{x\in X}$ (Theorem \ref{theo: characterize freeness through embedding problems}). Then $\f_x\colon \xymatrix{\Gamma \ar@{->>}^-{\beta_0}[r] & \Gamma(X) \ar@{->>}^{\f_x'}[r] & G}$ is a solution of (\ref{diagram: embedding problem}). Moreover, as $\prod_{x\in X}\f_x\colon\xymatrix{\Gamma \ar@{->>}^-{\beta_0}[r] & \Gamma(X) \ar@{->>}^-{\prod \f_x'}[r] & \prod_{x\in X}(G\twoheadrightarrow H)}$ is a quotient map, the family $(\f_x)_{x\in X}$ is independent. Thus $\Gamma$ is isomorphic to $\Gamma(X)$ by Theorem \ref{theo: characterize freeness through embedding problems}.
\end{proof}

We are now prepared to prove our main result.

\begin{theo} \label{theo: main}
	The differential Galois group of the family of all regular singular differential equations over $\C(x)$ is isomorphic to the free proalgebraic group on a set of cardinality $|\C|$.
\end{theo}
\begin{proof}
	Let $\Gamma=\Gamma_\C=G(L/\C(x))$ be the differential Galois group of the family of all regular singular differential equations.
	By \cite[Lem. 3.3]{BachmayrHarbaterHartmannWibmer:FreeDifferentialGaloisGroups}, the rank of $\Gamma$ is the smallest cardinal number $\kappa$ such that $L/\C(x)$ is a Picard-Vessiot extension for a family of differential equations of cardinality $\kappa$. Since the family of all differential equations over $\C(x)$ has cardinality $|\C|$, we have $\rank(\Gamma)\leq |\C|$.
	
	
	Let $H$ be an algebraic group and $\beta \colon \Gamma\twoheadrightarrow H$ a quotient map. By the differential Galois correspondence (Theorem \ref{theo: differential Galois correspondencee}) the extension $L^{\ker(\beta)}/\C(x)$ is a Picard-Vessiot extensions with differential Galois group $H$. Since $H$ is algebraic, $L^{\ker(\beta)}/\C(x)$ is a Picard-Vessiot extension of finite type. Therefore, there exists a finite family $\mathcal{F}$ of regular singular differential equations over $\C(x)$ such that the Picard-Vessiot extension $L_\mathcal{F}$ of $\C(x)$ for $\mathcal{F}$ (inside $L$) contains $L^{\ker(\beta)}$. Let $Y\subseteq \C$ be the finite set of (finite) singularities of the differential equations contained in $\mathcal{F}$. Let $X$ be a proper subset of $\C$ with $|X|=\C$ and $Y\subseteq X$ (e.g., $X=\C\smallsetminus\{x_0\}$ with $x_0\notin Y$).
	 Then $$\C(x)\subseteq L^{\ker(\beta)}\subseteq L_\mathcal{F}\subseteq L_Y\subseteq L_X$$ 
	and so, the quotient map $\beta\colon \Gamma\twoheadrightarrow H$ factors as $\beta\colon \Gamma\twoheadrightarrow \Gamma_X\twoheadrightarrow H$. As $\Gamma_X\cong\Gamma(X)$ by Proposition~\ref{prop: main for X}, we see that the condition of Lemma \ref{lemma: step for main theorem} is satisfied. 
\end{proof}
Combining Proposition \ref{prop: main for X} and Theorem \ref{theo: main} we obtain:

\begin{cor} \label{cor: main}
	Let $X$ be a subset of $\C$. Then the differential Galois group of the family of all regular singular differential equations with singularities in $X\cup\{\infty\}$ is isomorphic to the free proalgebraic group on $X$.
\end{cor}

\subsection{Open Questions} It seems natural to wonder if Corollary \ref{cor: main} remains valid when $\C$ is replaced with an algebraically closed field of characteristic zero. Our proof uses transcendental tools, such as the fundamental group, and therefore does not generalize. Even when $X$ is finite, it seems to be unknown if Corollary \ref{cor: main} is true for an algebraically closed field of characteristic zero in place of $\C$.

\bibliographystyle{alpha}
 \bibliography{bibdata}
\end{document}